\documentclass[12pt,leqno]{article}
\usepackage{amssymb}
\usepackage{amsmath}
\usepackage{amsfonts}
\usepackage{amsthm}
\newtheorem{corollary}{Corollary}[section]

\newtheorem{theorem}{Theorem}[section]
\newtheorem{lemma}{Lemma}[section]

\begin{document}
\title{On Dris Conjecture about Odd Perfect Numbers}
\author{Paolo Starni}%\\School of Economics, Management and Statistic\\ Rimini Campus\\ University of Bologna\\ Italy\\
%E-mail: paolo.starni@unibo.it}
\date{}

\maketitle

\begin{abstract}
The Euler's form of odd perfect numbers, if any, is $n=\pi^\alpha N^2$, where $\pi$ is prime, $(\pi,N)=1$ and  $\pi\equiv \alpha\equiv1\pmod{4}$. Dris conjecture  states that $N>\pi^{\alpha}$. We find that $N^2>\frac{1}{2}\pi^{\gamma}$, with $\gamma=max\{\omega(n)-1,\alpha\}$; $\omega(n)\geq 9$ is the number of distinct prime factors of $n$.
\end{abstract}

%section 1
\section{INTRODUCTION}
Without explicit definitions all the numbers considered in what follows must be taken as strictly positive integers. Let $\sigma(n)$ be the sum of the divisors of a number $n$; $n$ is said to be perfect if and only if $\sigma(n)=2n$. The multiplicative structure of odd perfect numbers, if any, is
\begin{equation}
n=\pi^\alpha N^2
\label{eq:1}
\end{equation}
where $\pi$ is prime, $\pi\equiv \alpha\equiv1\pmod{4}$ and $(\pi , N)=1$ (Euler, cited in \cite[p. 19]{Dickson}); $\pi^{\alpha}$ is called the Euler's factor. From equation (1) and from the fact that the $\sigma$ is multiplicative, it results also 
\begin{equation}
n=\frac{\sigma(\pi^{\alpha})}{2}\sigma(N^2)
\label{eq:2}
\end{equation}
where $\sigma(N^2)$ is odd and $2\Vert \sigma(\pi^{\alpha})$. Many details concerning the Euler's factor and $N^2$ are given, for example, in \cite{Starni1}\cite{Starni2}\cite{Starni3}\cite{MacDaniel}\cite{Shi}.
Regarding the relation between the magnitudo of $N^2$ and $\pi^{\alpha}$ it has been conjectured by Dris that $N>\pi^{\alpha}$ \cite{Dris}. The result obtained in this paper is \emph{a necessary condition for odd perfection} (Theorem 2.1) which provides an indication about Dris conjecture. Indicating with $\omega(n)$ the number of distinct prime factors of $n$, we prove that (Corollary 2.3):\\

$(i)$\hspace{1.1cm}$N^2>\frac{1}{2}\pi^{\gamma}$, where $\gamma=max\{\omega(n)-1, \alpha\}$ \\

\hspace{-0.6cm}Since $\omega(n)\geq9$ (Nielsen, \cite{Nielsen}), it follows:\\

$(i)_{1}$\hspace{0.9cm}$N^2>\frac{1}{2}\pi^8$; this improves the result $N>\pi$ claimed in \cite{Brown} by Brown in his approach to Dris conjecture.\\

\hspace{-0.6cm}Besides\\

$(i)_2$\hspace{0.7cm}    If $\omega(n)-1 >2\alpha$, then $N>\pi^{\alpha}$\\

\hspace{-0.6cm}so that \\

$(i)_3$\hspace{0.7cm} If $\omega(n)-1 > 2\alpha$ for each odd perfect number $n$, then Dris conjecture is true.\\

Now, some questions arise: $\omega(n)$ depends on $\alpha$? Is there a maximum value of $\alpha$? The minimum value of $\alpha$ is $1$? The only possible value of $\alpha$ is 1 (Sorli,  \cite[conjecture 2]{Sorli}) so that Dris conjecture is true? Without ever forgetting the main question: do odd perfect numbers exist?

%section 2

\section{THE PROOF}
Referring to an odd perfect number $n$ with the symbols used in equation \eqref{eq:1}, we obtain:
\begin{lemma}
If $n$ is an odd perfect number, then 
\begin{equation*}
N^2=A\frac{\sigma(\pi^{\alpha})}{2} \hspace{0.2cm}and\hspace{0.2cm} \sigma(N^2)= A\pi^{\alpha}
\end{equation*}
\label{lemma}
\end{lemma}
\begin{proof}
 From equation \eqref{eq:2} and from the fact that $(\sigma(\pi^{\alpha}),\pi^{\alpha})=1$, it follows
 \begin{equation}
 N^2=A\frac{\sigma(\pi^{\alpha})}{2}
 \end{equation}
 where $A$ is an odd positive integer given by
  \begin{equation}
  A=\frac{\sigma(N^2)}{\pi^{\alpha}}
  \end{equation}
\end{proof}
In relation to the odd parameter $A$ in Lemma 2.1, we give two further lemmas:
\begin{lemma}
If $A=1$, then
$\alpha\geq \omega(n)-1 \hspace{0.1cm}and\hspace{0.1cm} N^2>\frac{1}{2}\pi^{\alpha}$
\end{lemma}
\begin{proof}
Let $q_k, k=1, 2, ...,\omega(N)=\omega (N^2)$, are  the prime factors of $N^2$; from hypothesis and from (4) we have
\begin{equation*}
 \pi^{\alpha}=\sigma(N^2)=\sigma(\prod_{k=1}^{\omega(N)}q_k^{2\beta_k})=\prod_{k=1}^{\omega(N)}\sigma(q_k^{2\beta_k})=\prod_{k=1}^{\omega(N)}\pi^{\delta_k}
 \end{equation*}
 in which $\alpha=\sum_{k=1}^{\omega(N)}\delta_k\geq \sum_{k=1}^{\omega(N)}1_k=\omega(N)$.\\ 
 \vspace{0.05cm}\\
 Since $\omega(n)=\omega(N)+1$, it results
 \begin{equation*}
 \alpha\geq \omega(n)-1
 \end{equation*}
 Besides, from Equation (3) it follows
\begin{equation*} 
N^2=\frac{1}{2}\sigma(\pi^{\alpha})>\frac{1}{2}\pi^{\alpha}
\end{equation*}
\end{proof}
\begin{lemma}
If $A>1$, then
$N^2>\frac{3}{2}\pi^{\alpha}$
\end{lemma}
\begin{proof}
From Equation (3) it results $A\geq3$. Thus
\begin{equation*}
N^2\geq\frac{3}{2}\sigma(\pi^{\alpha})>\frac{3}{2}\pi^{\alpha}
\end{equation*}
\end{proof}
The following theorem summarizes a necessary condition for odd perfection.
\begin{theorem}
If $n$ is an odd perfect number, then
\begin{equation*}
(\neg a\land d)\lor(a\land b\land c)\lor (b\land c\land d)
\end{equation*}
\label{th1}
where: $a\cong (A=1)$,$\neg a\cong (A>1)$, $b\cong (\alpha\geq \omega(n) -1)$, $c\cong (N^2>\frac{1}{2}\pi^{\alpha})$, $d\cong (N^2>\frac{3}{2}\pi^{\alpha})$
\end{theorem}
\begin{proof}
We combine Lemmas 2.2 and 2.3 setting  
\begin{equation}
\left\{
\begin{array}
[c]{c} 
lemma\hspace{0.1cm}2.2: (a\implies b\land c)\\
\hspace{-0.35cm}lemma\hspace{0.1cm} 2.3: (\neg a\implies d)
\end{array}  
\right.
\end{equation}
where, since it cannot be $A<1$,  it is $ (a)\cong (A=1)$ and $(\neg a)\cong (A>1)$. One obtains from (5)
\begin{equation*}
[\neg a\lor(b\land c)]\land (a\lor d)
\end{equation*}
which is equivalent to
\begin{equation}
(\neg a\land d)\lor (a\land b \land c)\lor (b\land c \land d)
\end{equation}
\end{proof}

Considering cases in which the necessary condition for odd perfection (6) is false, we obtain the following corollaries:
\begin{corollary}
If $n$ is an odd perfect number, then $N^2>\frac{1}{2}\pi^{\alpha}$
\end{corollary}
\begin{proof}
We have\\

\hspace{-0.6cm}(7)\hspace{0.3cm} $(\neg c\land \neg d) (\cong\ N^2<\frac{1}{2}\pi^{\alpha})\implies n$ is not an odd perfect number\\

\hspace{-0.6cm}From the contrapositive formulation of (7) it follows the proof. 
\end{proof}
\begin{corollary}
If $n$ is an odd perfect number, then
\begin{equation*}
N^2>\frac{3}{2}\pi^{\omega(n)-1}>\frac{1}{2}\pi^{\omega(n)-1}
\end{equation*}
\end{corollary}
\begin{proof}
We have\\

\hspace{-0.6cm}(8)\hspace{0.3cm} $(\neg b\land \neg d) (\cong\ N^2<\frac{3}{2}\pi^{\omega(n)-1})\implies n$ is not an odd perfect number\\

\hspace{-0.6cm}From the contrapositive formulation of (8) it follows the proof. 
\end{proof}
Combining these two corollaries, we have
\begin{corollary}
If $n$ is an odd perfect number, then
\begin{equation*}
N^2>\frac{1}{2}\pi^{\gamma}, where\hspace{0.1cm} \gamma=max\{ \omega(n)-1,\alpha\}
\end{equation*}
\end{corollary}
\begin{proof}
Immediate.
\end{proof}

%\newpage
%\vspace{0.5cm}

%\begin{acknowledgement}
%In case of publication of the article 
%\end{acknowledgement}

\end{document}